\documentclass[letterpaper, 10 pt, conference]{ieeeconf}  

\IEEEoverridecommandlockouts                              

\usepackage{graphicx}
\usepackage{amsmath}
\usepackage{algorithm}
\usepackage{amsfonts}
\usepackage{listings}
\usepackage{lastpage}
\usepackage{enumerate}
\usepackage{fancyhdr}
\usepackage{mathrsfs}
\usepackage{xcolor}
\usepackage{graphicx}
\usepackage{listings}
\usepackage[hidelinks]{hyperref}
\usepackage{pythonhighlight}
\usepackage[utf8]{inputenc}
\usepackage{algpseudocode}

\newtheorem{proposition}{Proposition}
\newtheorem{theorem}{Theorem}

\title{\LARGE \bf
A Robust Mean-field Game of Boltzmann-Vlasov-like Traffic Flow
}

\author{Amoolya Tirumalai$^{1}$ and John S. Baras$^{1}$
\thanks{*This work was supported in part by US Office of Naval Research (ONR) Grant No. N00014-17-1-2622.}
\thanks{$^{1}$A. Tirumalai and J.S. Baras are with the Department of Electrical and Computer Engineering and the 
Institute of Systems Research at the University of Maryland, 8223 Paint Branch Dr, College Park, MD 20740, USA.
        Email: {\tt\small \{ast256, baras\}@umd.edu}.}%
}

\begin{document}

\maketitle
\thispagestyle{empty}
\pagestyle{empty}

\begin{abstract}

Historically, traffic modelling approaches have taken either a particle-like (microscopic) approach, or a gas-like (meso- or macroscopic) approach. Until recently with the introduction of mean-field games to 
the controls community, there has not been a rigorous framework to facilitate passage between controls for the microscopic models and the macroscopic models. We begin this work with a particle-based model of autonomous vehicles subject to 
drag and unknown disturbances, noise, and a speed limit in addition to the control. 

 We formulate a robust stochastic differential game on the particles. We pass formally to the infinite-particle limit to obtain a robust mean-field game PDE system. We solve the mean-field game PDE system numerically and discuss the results. In particular, we obtain an optimal control which increases the bulk velocity of the traffic flow while reducing congestion.

\end{abstract}

\section{Introduction}
Models of traffic flow were developed first in the 1950s with the Lighthill-Whitham model \cite{lighthill1955kinematic}, which is generally a non-linear hyperbolic PDE for the spatial vehicle density. Similarly to the canonical Burgers' equation of fluid dynamics, the Lighthill-Whitham model exhibits phenomena of shock wave formation corresponding to the formation of traffic jams. This is a variety of macroscopic model.

Generally, these models are constructed from macroscopic conservation laws. While these methods are convenient and capture much of the fundamental qualties of the traffic flow, starting and ending with macroscopic models of traffic flow makes it rather difficult to map controls back to the vehicles.

To solve this problem, we begin with microscopic agent dynamics.  The particular model chosen is that of a particle subject to fluid drag, control, disturbance, and random excitations. A speed limit, actuation constraints, and disturbance constraints are imposed.  

A robust stochastic differential game is formulated on these particle dynamics. The cost for this game accounts for passenger comfort and dissipation of congestion at the highest velocity possible. In particular, the congestion term depends on the empirical measure of the $N$-driver system. Some inspiration is taken from \cite{chevalier2015micro}, but we have expanded the problem therein to include noise and robustness, as well as higher-order dynamics.

We (formally as opposed to rigorously) pass to an infinite-driver limit of this problem, and obtain a mean-field game. Via stochastic dynamic programming, the robust mean-field game is solved by a forward Kolmogorov equation and a backward Hamilton-Jacobi-Bellman-Isaacs equation. In fact, the forward Kolmogorov equation we obtain corresponds to a Boltzmann-Vlasov-like equation. Similar equations were developed for traffic flow modelling in the 1960s \cite{prigogine1960boltzmann}. For the rest of the paper, we present numerical methods to solve this backward-forward system of PDEs, and a numerical example. 
\section{Notation}
In this paper, we define $\mathbb T:= \mathbb R / 200\pi \mathbb Z$, which is the standard circle streched from $[0,2\pi)$ to $[0,200 \pi)$. For $(x_1, x_2, x_3) \in \mathbb T \times \mathbb R \times \mathbb R^+_0$, 
$\partial_{x_i}(\cdot)$ denotes partial differentiation w.r.t. the subscripted variable, and $\nabla_{(x_i, x_j)}(\cdot)$ represents the gradient w.r.t. the given variables. $\nabla_{(x_i,x_j)}\cdot (\cdot) $ is the divergence w.r.t. those variables. For a Polish space $X$, $\mathcal B(X)$ is the Borel $\sigma$-algebra associated to that space, and $\{ X, \mathcal B(X) \}$ the associated Borel space. $\mathcal P(X)$ is the space of all probability distributions (measures) on $X$. For $A \in \mathcal B(X)$, $\mathbb I_A(\cdot)$ is the indicator function of $A$, defined in the usual way. Given a filtration $\mathcal F(t) \subset \mathcal B(X)$, $L^2_{\mathcal F}(0,T;L^2(\Omega; X))$ is the set of square-integrable random processes taking values in Polish space $X$ which are also adapted to the filtration $\mathcal F(\cdot)$ at each time $t \in [0,T]$.
\section{Problem Formulation}

Take a standard Borel space $\{\Omega, \mathcal B(\Omega) \}$.
Consider a sequence of agents' positions-velocity pairs: $\{(x^i_t, v^i_t )\}_{i=1}^N \subset 
(\mathbb T \times [0,s_{max}])$, with $(x_{(\cdot)}^i, v_{(\cdot)}^i):\Omega \times [0,T]\rightarrow \mathbb T \times [0,s_{max}]$. Take these to follow the reflected It\^o SDEs (resulting in a degenerate reflected diffusion process):
\begin{equation}
\begin{split}
dx_t^i &= v_t^i dt \\ 
dv_t^i &= (-(v_t^{i})^2 \alpha + w^i(t)) dt + dr_t^i + \sqrt{2\epsilon}dW_t^i\\
&(x^i_t,v^i_t)|_{t = 0} = (x_0, v_0) \in L^2(\Omega;\mathbb T \times [0,s_{max}]).
\end{split} 
\end{equation}
$\alpha > 0$ is a parameter combining the drag coefficient, air density, vehicle cross sectional area, and vehicle mass. Assume that the autonomous vehicles are of a similar class, so we apply a common $\alpha$. $W^i_{(\cdot)}$ is the standard Wiener process \cite{oksendal2010stochastic}, of which we have $N$ independent copies. Note that the dynamics for the agents are independent of each other. Let $\mathcal F^i(t):= \sigma((x_0,v_0,W^i_s)  | s \leq t),$ which is the filtration generated by the initial conditions $(x_0, v_0)$ and the Wiener process $W^i_{(\cdot)}$. $r_t^i \in L^2_{\mathcal F^i}(0,T;L^2(\Omega; \mathbb R))$ is a regulator process whose increment is:
$$
dr_t^i = u^i(t)dt + \nu(v^i_t)dl_t^i,
$$
where $\nu(0) = 1, \nu(s_{max}) = -1, \nu(v) = 0$ otherwise on $[0,s_{max}]$. This is the inward pointing direction. $l_t^i \in L^2_{\mathcal F^i}(0,T;L^2(\Omega; \mathbb R))$ is a surely non-decreasing jump process (see \cite{tanaka1979stochastic, lions1984stochastic, pilipenko2014introduction} for details) with $l_{t=0}^i = 0$, and
$$
\int_0^T \mathbb I_{(0,s_{max})}(v_t^i) dl_t^i = 0,
$$
i.e. jumps only occur when the process $v_t^i$ is on the boundary of $[0,s_{max}]$. This has the effect of enforcing the speed limit. $u^i \in \mathcal U :=
PWC([0,T];[u_{min},u_{max}])$ is the acceleration or brake input supplied by the controller, and $w^i \in \mathcal W :=
PWC([0,T];[-w_{max},w_{max}])$ is a bounded disturbance representing external load. Suppose $$w_{max} < \min(\{|u_{min}|,u_{max} \}), u_{min} < 0 < u_{max}.$$

By the given assumptions, invoking the classical results of Tanaka \cite{tanaka1979stochastic}, there is a unique strong solution $(x_t^i, v_t^i) \in L^2_{\mathcal F^i}(0,T;L^2(\Omega; \mathbb T \times [0,s_{max}]))$. Indeed, it is actually $L^4_{\mathcal F^i}(\Omega;C([0,T];\mathbb T \times [0,s_{max}]))$ following the result of Lions and Sznitman \cite{lions1984stochastic}, which is considerably stronger.

Define the empirical measures (random measures) of sequences $\{x^j_t, v^j_t \}_{j=1\neq i}^N$ given by $\mu_{(\cdot)}^{N,i} : \Omega \times [0,T] \rightarrow 
\mathcal P(\mathbb T \times [0,s_{max}])$ (omitting dependence on $\omega \in \Omega$):
$$
\mu_t^{N,i}(A):= \frac{1}{N-1}\sum_{j=1\neq i}^N \delta_{(x_t^j,v_t^j)}(A), 
A \in \mathcal B(\mathbb T \times [0,s_{max}])
$$
where $\delta_{(\cdot)}$ is the standard Dirac measure on $\mathbb T \times [0,s_{\max}]$. Let $x_t = (x_t^1, ..., x_t^N)^\top, v_t = (v_t^1, ..., 
v_t^N)^\top$, and denote the exclusion of the $i$-th entries of each of these vectors by $(x_t^{-i},v_t^{-i})$. Let the controls taken by these agents be $u^{-i}$, and their associated disturbances be $w^{-i}$. Assume $w^i$ is unknown to the controller. Suppose that the optimal control taken by the `$-i$' agents exists and is $\hat{u}^{-i}$, and that the worst-case disturbance for the `$-i$' agents exists and is $\hat{w}^{-i}$.

For each agent, define the following robust optimal control problems (a differential game \cite{basar2018handbook}):
\begin{equation}
\begin{split}
\min_{u^i \in \mathcal U} \max_{w^i \in \mathcal W}&\text{ }   \mathbb E \Big [ \mathcal J_i[u^i,w^i, \hat{u}^{-i}, \hat{w}^{-i};(x_{t=0}, v_{t=0}), 0, T] \Big ], \\ 
&\text{s.t.} \text{ } (1).
\end{split}
\end{equation}

 The cost functional is:
\begin{equation*}
\begin{split}
\mathcal J_i[u^i,w^i, \hat{u}^{-i}, \hat{w}^{-i};(x_{t}, v_{t}), t, T]:= ... \\
\int_{t}^T \mathcal L_i(x_s, v^i_s, u^i(s), w^i(s))ds,
\end{split}
\end{equation*}
where
\begin{equation*}
\begin{split}
&\mathcal L_i(x,v^i, u^i,w^i):= \frac{1}{2}(u^i)^2 -  \frac{1}{2\gamma^2} (w^i)^2 + ... \\ &(\sum_{j=1\neq i}^N \phi(x^i,x^j)\mu_t^{N,i}(x_j, v_j) - \beta^{-1})v^i.
\end{split}
\end{equation*}

The quadratic term in $u^i$ is included to represent ride comfort, the $w^i$ term is included to reward the controller for resisting the disturbance, the term containing the empirical measure and $\phi \in C_b(\mathbb T^2;\mathbb R)$ represents an aversion to congestion, and the inclusion of the multiplication of $v$ by $\beta^{-1}$ enforces a preference to move as quickly as possible. 

We could apply dynamic programming to each agent's optimal control problem, but this would lead to a system of coupled dynamic programming  equations, and, assuming that the solution to such a system exists, this would lead to a fully centralized optimal control \cite{basar2018handbook, friedman1972stochastic}. 
\section{Mean-field Game}
To circumvent such a scenario, we'll formulate a mean-field game \cite{bensoussan2013mean}. Assume that there is some 
$m(\cdot) \in C([0,T];\mathcal P(\mathbb T \times \mathbb [0,s_{max}]))$ s.t. for each $t \in [0,T]$, $i=1,...,N$, $\mu_t^{N,i} \rightarrow m(t)$ weakly$^*$ in 
$\mathcal P(\mathbb T \times [0,s_{max}])$, which is the mean-field measure. Exclude a single agent from the mean-field. This (anonymous) agent follows the It\^o SDE:
\begin{equation}
\begin{split}
dX_t &= V_t dt \\ 
dV_t &= [-(V_t)^2 \alpha +  w(t)]dt + dr_t + \sqrt{2\epsilon}dW_t\\
&(X_t,V_t)|_{t = 0} = (x_0, v_0) \in L^2(\Omega;\mathbb T \times [0,s_{max}]).
\end{split} 
\end{equation}
where $r_t$ is of the same form as $r_t^i$, but we have dropped the indexing.
Assume that $u,w$ are of the given class so there is a unique $L^2_{\mathcal F}(0,T;L^2(\Omega; \mathbb T \times [0,s_{max}]))$ solution, where $\mathcal F(t):= \sigma((x_0,v_0,W_s)  | s \leq t)$ is the filtration generated by the initial conditions and the anonymized scalar Wiener process. 
Suppose the optimal control followed by the exogenous agents is $u^*$ and the worst-case disturbance they are subject to is $w^*$. The exogenous (anonymized) agents each follow copies of the It\^o SDE:
\begin{equation}
\begin{split}
d\Xi_t &= \Upsilon_t dt \\ 
d\Upsilon_t &= [-(\Upsilon_t)^2 \alpha + w^*(\Xi_t, \Upsilon_t,t)]dt + dR_t + \sqrt{2\epsilon} dB_t \\
&(\Xi_t,\Upsilon_t)|_{t = 0} = (x_0, v_0) \in L^2(\Omega;\mathbb T \times [0,s_{max}]).
\end{split} 
\end{equation}
where $B_{(\cdot)}$ is an anonymized scalar Wiener process, and:
$$
dR_t = u^*(\Xi_t, \Upsilon_t, t)dt + \nu(\Upsilon_t) dL_t,
$$
where $L_t$ is a jump process that has the same properties as $l_t, l_t^i$ taken w.r.t. $\Upsilon_t$. Let
$\mathcal F^*(t):= \sigma((x_0,v_0,B_s)  | s \leq t).$
Assume that $u^*, w^*$ exist and are each $\text{Lip}(\mathbb T \times [0,v_{max}];\mathbb R)$. Again, via the result of Tanaka \cite{tanaka1979stochastic}, there is a unique strong solution $(\Xi_{(\cdot)}, \Upsilon_{(\cdot)}) \in L^2_{\mathcal F^*}(0,T;L^2(\Omega; \mathbb T \times [0,s_{max}]))$, and the stronger result of Lions-Sznitman \cite{lions1984stochastic} also applies.

Let $m(t)$ be the distribution (law) of $(\Xi_t, \Upsilon_t)$. We can show that $m$ satisfies the (degenerate) forward Kolmogorov (FK) equation (in the sense of measures) over $D^m_T:= (0,T] 
\times \mathbb T \times [0,s_{max}]$:
\begin{equation}
\begin{split}
\label{measures}
&\partial_t m + \nabla_{\xi,\upsilon} \cdot [(\upsilon, -\alpha \upsilon^2 + u^*+ w^*)  m ] =  \epsilon \partial_\upsilon^2 m, \\
& m(0) = m_0.
\end{split} 
\end{equation}
We can recognize this as a Boltzmann-Vlasov-type kinetic equation with diffusion in the velocity \cite{neunzert1984introduction}, albeit a measure-valued version.

Since we have assumed $\mu_t^{N,i} \rightarrow m(t)$ weakly$^*$ in 
$\mathcal P(\mathbb T \times [0,s_{max}])$ for each $t \in [0,T]$ and each $i \in \{1,...,N \}$, $\phi \in C_b(\mathbb T^2;\mathbb R)$, and $\mathbb T$ is compact Hausdorff, the cost functional $\mathcal L_i 
\rightarrow \mathcal L$, and:
\begin{equation*}
\begin{split}
&\Lambda(\xi,\upsilon,u,w, m) := 
\mathcal L(\xi,\upsilon, u,w):= \frac{1}{2}u^2 - \frac{1}{2\gamma^2}w^2 + ... \\ 
&(\int_{\mathbb T \times [0,v_{max}]} \phi(\xi,s) dm(s,w)(t) - \beta^{-1})\upsilon .
\end{split}
\end{equation*}
We now have a mean-field game:
\begin{equation}
\begin{split}
\min_{u \in \mathcal U} \max_{w \in \mathcal W}&\text{ }   \mathbb E \Big [ \mathcal J[u,w, {u}^{*}, {w}^{*};(X_{t=0}, V_{t=0}), t=0, T] \Big ] \\ 
&\text{s.t.} \text{ } (3), \\
\end{split}
\end{equation}
with
\begin{equation}
\begin{split}
\mathcal J[u,w, {u}^{*}, {w}^{*};(X_{t}, V_{t}), t,T]:= ... \\
\int_{t}^T \Lambda(X_s,V_s, u(s), w(s), m(s)) ds.
\end{split}
\end{equation}
\begin{proposition}
Let $D_T^{\mathcal V}:= [0,T) \times \mathbb T \times [0, s_{max}]$, 
$$\mathbf U:= [u_{min},u_{max}], \mathbf W:=[-w_{max},w_{max}].$$
Define the pre-Hamiltonian:
$$
\mathcal H(\xi, \upsilon, u, w, m,p):= \Lambda(\xi, \upsilon, u, w, m) + p^\top g(\xi, \upsilon, u, w),
$$
where $g(\xi, \upsilon, u, w):= (\upsilon, -\alpha \upsilon^2 + u + w)^\top$.
The value function and optimal measure  $(\mathcal V, m)$ corresponding to the solution of the given robust mean-field game satisfy the system of PDEs (suppressing some arguments):
\begin{equation}
\begin{split}
&\partial_t \mathcal V + H(\cdots, m, \nabla_{\xi,\upsilon}\mathcal V) + \epsilon \partial_{\upsilon}^2 \mathcal V = 0 \text{ in } D_T^{\mathcal V}; \\
&H(\xi,\upsilon,m,p):=  \min_{u \in \mathbf U} \max_{w \in \mathbf W } \mathcal H(\xi, \upsilon, u, w, m,p); \\
& \partial_t m+ \nabla_{\xi,\upsilon} \cdot \Big [(\upsilon, -\alpha \upsilon^2 + u^*+ w^*) m \Big ] = \epsilon \partial_\upsilon^2 m \text{ in } D_T^m ;\\ 
& w^*(\xi,\upsilon,u) = \arg \max_{w \in \mathbf W} \mathcal H(\xi, \upsilon, u, w, m,\nabla_{\xi,\upsilon}\mathcal V) ; \\ 
&u^*(\xi,\upsilon)  = \arg \min_{u \in \mathbf U} \mathcal H(\xi, \upsilon, u, w^*(\cdots,u), m,\nabla_{\xi,\upsilon}\mathcal V); \\
& \mathcal V (t,\xi,\upsilon) = \mathcal V (t,\xi + 200\pi,\upsilon) \\
& \partial_\upsilon \mathcal V (t,\xi,0) = \partial_\upsilon \mathcal V (t,\xi,s_{max}) = 0 \\
&\mathcal V(T,\cdot) = 0, m(0) = m_0 \in \mathcal P(\mathbb T \times [0, s_{max}]).
\end{split}
\end{equation}
assuming a solution exists and is regular-enough.
\end{proposition}
\begin{proof}
As we have employed a time-consistent formulation of the mean-field game \cite{bensoussan2013mean}, we apply the standard procedure of stochastic dynamic programming  (with a slightly different version of It\^o's lemma \cite{watanabe1971stochastic}) to obtain (in addition to (\ref{measures})) the Hamilton-Jacobi-Bellman-Isaacs (HJB-I) equation:
\begin{equation}
\begin{split}
\partial_t \mathcal V + H(\xi,\upsilon, m, \nabla_{\xi,\upsilon}\mathcal V) + 
\epsilon \partial_\upsilon^2 \mathcal V = 0 \text{ in }  D_T^{\mathcal V};
\end{split}
\end{equation}
\begin{equation}
\begin{split}
H(\xi,\upsilon,m,p):= &\min_{u \in \mathbf U} \max_{w \in \mathbf W} \mathcal H(\xi, \upsilon, u, w, m,p),
\end{split}
\end{equation}
subject to the given boundary conditions. See \cite{watanabe1971stochastic} for justification of the Neumann conditions in the $\upsilon$ co-ordinate.

The optimal control and worst-case disturbance are:
$$
u^*(p_2):= \begin{cases} 
-p_2 &\text{ if } -p_2 \in [u_{min}, u_{max}] \\
 \bar u &\text{ else },
\end{cases}
$$
$$
w^*(p_2):= \begin{cases} 
\gamma^2 p_2 &\text{ if } \gamma^2 p_2 \in [-w_{max}, w_{max}] \\
\bar w &\text{ else },
\end{cases}
$$
where $$\bar u := \arg \min_{u \in \{u_{min}, u_{max} \}} \mathcal H(\cdots),$$ $$\bar w:= 
\arg \max_{w \in \{-w_{max},w_{max} \}} \mathcal H(\cdots),$$
which can be found by computing the parts of $\mathcal H$ which depend on $u,w$ on the boundaries of the control and disturbance sets and comparing the values. As there are only two points to check for each (being scalars), this computation is trivial.
\end{proof}

There are results on existence and uniqueness for the mean-field game system in the deterministic state-constrained and periodic cases separately \cite{cannarsa2021mean}, and in the stochastic state-constrained non-degenerate case \cite{bensoussan2013mean}. To the knowledge of the authors, there are no results for the (robust) mean-field game system with mixed boundary conditions and degenerate diffusion. For now, we proceed assuming a suitable solution exists, and intend to explore this problem further.
\section{Numerical Solution of the HJB-I-FK System}
Suppose that the density of $m(\cdot)$ w.r.t. the Lebesgue measure is $\rho:[0,T] \times \mathbb T \times [0,s_{max}] \rightarrow \mathbb R^+_0$. From the measure-valued formulation of the forward Kolmogorov equation for $m$, we obtain:
\begin{equation}
\begin{split}
\label{measures}
&\partial_t \rho + \nabla_{\xi,\upsilon} \cdot [(\upsilon, -\alpha \upsilon^2 + u^*+ w^*)  \rho ] =  \epsilon \partial_\upsilon^2 \rho, \\
& \rho(t, \xi,\upsilon) = \rho(t, \xi + 200 \pi,\upsilon), \\
& \partial_\upsilon \rho(t,\xi,0) = \partial_\upsilon \rho(t,\xi,s_{max}) = 0 \\
& \rho(0,\cdot) = \rho_0 := dm_0/d\lambda, \int_{\mathbb T \times [0,s_{max}]} \rho(t,\cdot) d\xi d\upsilon = 1
\end{split} 
\end{equation}
which is, of course, to be interpreted in the weak formulation. This equation together with HJB-I are the equations whose solutions we approximate using the procedure below.  Define a grid of points (with the same number of points in each direction) $\{(\xi_i,\upsilon_j) \}_{i,j=1}^{N_x,N_x}$, where 
\begin{equation}
\begin{split}
&\xi_{i+1} - \xi_i = h, \upsilon_{j+1} - \upsilon_{j} = k, \\
&\xi_0 = 0, \xi_{N_x-1} = 200 \pi - h, \\
&\upsilon_0 = k/2, \upsilon_{N_x-1} = s_{max} - k/2
\end{split} 
\end{equation} 

Define sequence of cells of area $hk$ centered on $(\xi_i, \upsilon_j)$ called $\{I_{i,j}\}_{i,j=1}^{N_x,N_x}$. These cells are staggered in the $\upsilon$ direction, and unstaggered in $\xi$. Before we proceed, let us also define the following difference operators for a function $f_{i,j} := f(\xi_i,\upsilon_j)$ defined on the given grid of points:
$$
(D_1^{+}f)_{i,j}:= \frac{f_{i+1,j} - f_{i,j}}{h}, (D_2^{+}f)_{i,j}:= \frac{f_{i,j+1} - f_{i,j}}{k}
$$
which is the upwind difference operator, and similarly:
$$
(D_1^{-}f)_{i,j}:= \frac{f_{i,j} - f_{i-1,j}}{h}, (D_2^{-}f)_{i,j}:= \frac{f_{i,j} - f_{i,j-1}}{k}
$$
which is the downwind operator. We also define:
\begin{equation}
\label{sod}
(D^2_2 f)_{i,j}:= \frac{(D_2^{+}f)_{i,j} - (D_2^{-}f)_{i,j}}{k}
\end{equation}
which is the centered second-order difference operator. Let the approximate value function be $\hat{\mathcal V}_{i,j}(t):= \hat{\mathcal V}(t,\xi_i, \upsilon_j),$ and the approximate density be $\hat \rho_{i,j}(t) =  \hat \rho(t, \xi_i, \upsilon_j)$. We first dicretize space to obtain a semi-discrete system, and then briefly describe the time-discretization and fixed-point procedure to solve the backward-forward PDE system.
\subsection{Finite Difference Approximation of HJB-I}
To approximate the solution of the HJB-I equation, we must discretize the Hamiltonian $H$ in a way that is consistent with its continuous properties. Let $\mathbf H$ be the numerical Hamiltonian. The choice of $\mathbf H$ should satisfy the monotonicity, consistency, and regularity assumptions as in \cite{osher1991high, achdou2020mean}. The first-order upwind Hamiltonian satisfies the desired properties \cite{botkin2011stable, achdou2010mean}, and is used in a similar problem to our own \cite{chevalier2015micro}. Suppose $\hat m(\cdot)$ has the approximate density $\hat \rho$ w.r.t. the Lebesgue measure. The upwind numerical Hamiltonian is:
$$
\mathbf H_{i,j}(t) = H(\xi_i, \upsilon_j, \hat m(t), (D^+_1 \hat{\mathcal V}, D^+_2 \hat{\mathcal V})_{i,j}(t)).
$$
Combining this with the second-order difference formula, we obtain the semi-discrete scheme for the HJB-I equation:
\begin{equation}
\frac{d}{dt}\hat{\mathcal V}_{i,j}(t) = - \mathbf H_{i,j}(t) - \epsilon (D^2_2 \hat{\mathcal V})_{i,j}(t)
\end{equation}
subject to the terminal condition $\hat{\mathcal V}_{ij}(T) = 0$. The periodic and homogeneous Neumann conditions are implemented in the usual way using ghost nodes \cite{leveque2002finite, achdou2020mean}. This is solved backward in time.
\subsection{Finite Volume Approximation of FK}
As with the HJB-I equation, there are many ways to discretize the forward Kolmogorov equation.  We desire a method which preserves monotonicity and non-negativity of the discretized probability density. The second requirement we consider to be particularly important. We employ the first-order Rusanov method on the hyperbolic part, which satisfies the desired properties \cite{zhang2017positivity}. On the parabolic part, we use (\ref{sod}). We integrate the forward Kolmogorov equation over $I_{i,j}$:
\begin{equation}
\begin{split}
&\int_{I_{i,j}} \partial_t \rho d\xi d\upsilon = ... \\
& \int_{I_{i,j}} - \nabla_{\xi,\upsilon} \cdot [(\upsilon, -\alpha \upsilon^2 + u^* + w^*)\rho ] + \epsilon \partial_v^2 \rho \text{ }d\xi d\upsilon. 
\end{split}
\end{equation}
Now, we will approximate the function $\rho$ by $\hat \rho$ which is assumed to be piecewise constant in each $I_{i,j}$. The equation now becomes:
\begin{equation}
\begin{split}
\label{flux}
& \frac{d}{dt}\hat \rho_{i,j}(t) = - \frac{1}{h}[g^1_{i+1/2,j}(\hat \rho(t)) - g^1_{i-1/2,j}(\hat \rho(t))] - ... \\ 
& \frac{1}{k}[g^2_{i,j+1/2}(\hat \rho(t)) - g^2_{i,j-1/2}(\hat \rho(t))] + \epsilon (D^2_2 \hat \rho(t))_{i,j}.
\end{split}
\end{equation}
subject to the initial condition: $\hat \rho_{ij}(0) = \rho_0(\xi_i,\upsilon_j)$. This is solved forward in time.
The Rusanov (local Lax-Friedrichs) fluxes \cite{leveque2002finite} are:
$$
g^1_{i \pm 1/2, j}(\hat \rho):= \frac{1}{2}(\upsilon_j\hat \rho_{i,j} + \upsilon_j\hat \rho_{i \pm 1,j} - |\upsilon_j|(D^\pm_1 (\rho))_{i,j})
$$
\begin{equation*}
\begin{split}
g^2_{i , j\pm 1/2}(\hat \rho)&:= \frac{1}{2}(\psi_{i,j}\hat \rho_{i,j} + \psi_{i, j \pm 1}\hat \rho_{i ,j \pm 1 } - ... \\ 
&max_{\pm}(|\psi_{i, j \pm 1}| ) (D^\pm_1 (\hat \rho))_{i,j}),
\end{split}
\end{equation*}
where
$$
\psi_{i,j} = \psi(\xi_i,\upsilon_j,\ u^*_{i,j}, w^*_{i,j}):= -\alpha \upsilon_j^2 +u^*_{i,j} + w^*_{i,j}.
$$
We have the following result:
\begin{proposition}
Consider the grid of points previously defined $\{(\xi_i, \upsilon_j)\}_{i,j=1}^{N_x, N_x}$, and append the ghost points $\{(\xi_i, \upsilon_{0}) \}_{i=1}^{N_x}$, $\{(\xi_i, \upsilon_{N_x + 1} )\}_{i=1}^{N_x}$, $\{(\xi_i, \upsilon_{0} )\}_{i=1}^{N_x}$, $\{(\xi_i, \upsilon_{N+1} )\}_{i=1}^{N_x}$ to the grid, where $\xi_0 = -h$, $\xi_{N_x + 1} = 200\pi$, 
$\upsilon_0 = -k/2$, $\upsilon_{N_x + 1} = s_{max} + k/2$. If $\hat \rho$ is extended to the ghost nodes in $i$ as:
$$
\hat \rho_{0,j}(t) = \hat \rho_{N_x,j}(t), \hat \rho_{N_x + 1, j}(t) = \hat \rho_{1, j}(t),
$$
which is the extension corresponding to the assumption: $\hat \rho_{i + N_x, j}(t) = \hat \rho_{i,j}(t)$, and extended to the ghost nodes in $j$ as:
$$
\hat \rho_{i,0}(t) = \hat \rho_{i,1}(t), \hat \rho_{i, N_x + 1}(t) = \hat \rho_{i, N_x}(t),
$$
which correspond to even extensions about $\upsilon_{1/2} = 0$, and  $\upsilon_{N_x+1/2} = s_{max}$, and if:
$$
\psi_{i,0} = - \psi_{i,1}, \psi_{i, N_x + 1} = - \psi_{i, N_x},
$$
then, the total mass is conserved:
$$
\frac{d}{dt}\sum_{i,j=1}^{N_x,N_x} \hat \rho_{i,j}(t) = 0.
$$
\end{proposition}
\begin{proof}
The Rusanov flux is well-known to be conservative \cite{leveque2002finite}. Applying the given sum to the RHS of (\ref{flux}) yields:
\begin{equation*}
\begin{split}
&\frac{d}{dt}\sum_{i,j=1}^{N_x,N_x} \hat \rho_{i,j}(t)  = \sum_{j=1}^{N_x} -\frac{1}{h}[g^1_{1/2, j} + g^1_{N_x + 1/2, j}] + ... \\
& \sum_{i=1}^{N_x} -\frac{1}{k}[g^2_{i, 1/2} - g^2_{i, N_x + 1/2}] + ... \\
&\frac{1}{k}\Big [\frac{\hat \rho_{i,0} - \hat \rho_{i,1}}{k} - 
\frac{\hat \rho_{i,N_x + 1} - \hat \rho_{i,N_x}}{k} \Big ] .
\end{split}
\end{equation*}
Immediately, from the selection of the values of $\hat \rho$ on the ghost nodes,  the last term drops out, and if we inspect the structure of $g^1_{(\cdot),j}$, we see clearly that the terms containing it also drop out from the selection of the ghost values of $\hat \rho$. With respect to the terms containing $g^2_{i, (\cdot)}$, we inspect its definition and conclude that the final term drops out for both of the grid points due to the ghost values of $\hat \rho$. So, we are left with:
\begin{equation*}
\begin{split}
&\frac{d}{dt}\sum_{i,j=1}^{N_x,N_x} \hat \rho_{i,j}(t)  = \sum_{i=1}^{N_x} \Big (-\frac{1}{k}[g^2_{i, 1/2} - g^2_{i, N_x + 1/2}] \Big ) \\
& = \sum_{i=1}^{N_x} -\frac{1}{2k}[(\psi_{i, 0}\hat \rho_{i, 0} + \psi_{i,1}\hat \rho_{i, 1}) - ... \\ 
&(\psi_{i, N_x}\hat \rho_{i, N_x} + \psi_{i,N_x + 1}\hat \rho_{i, N_x + 1})] \\
&=  \sum_{i=1}^{N_x} -\frac{1}{2k}[(-\psi_{i, 1}\hat \rho_{i, 1} + \psi_{i,1}\hat \rho_{i, 1}) - ... \\ 
&(\psi_{i, N_x}\hat \rho_{i, N_x} - \psi_{i,N_x}\hat \rho_{i, N_x})] = 0.
\end{split}
\end{equation*}
Thus, the claim is proven.
\end{proof}

\subsection{Iterative Solution of Backward-Forward System}
\begin{algorithm}
\caption{Backward-Forward Iteration}
\begin{algorithmic}
\label{fwd_bkwd}
\Require $\rho_0, \mathcal V_T, n_{iters}$
\Ensure $\hat{\mathcal V}^*, \hat{\rho}^*$
\For {$\theta=1,N_T$}
	\State $\hat{\rho}^{-1}(t_\theta,\cdot) \gets \rho_0$
\EndFor
\State $\hat{\mathcal V}^0 \gets \text{solveHJB-I}(\hat \rho^{-1}, \mathcal V_T)$
\State $\hat \rho^{0} \gets \text{solveFK}(\hat{\mathcal V^{0}}, \rho_0)$
\State $D \gets 0$
\For{$n = 1, n_{iters}$}
\State $\hat{\mathcal V}^{n} \gets \text{solveHJB-I}(\hat \rho^{n-1}, \mathcal V_T)$
\State $\hat \rho^{n} \gets \text{solveFK}(\hat{\mathcal V^{n}}, \rho_0)$
\State $\delta^2_n \gets  \text{computeDelta2}(\hat{\mathcal V}^n, \hat{\rho}^n, \hat{\mathcal V}^{n-1}, \hat{\rho}^{n-1})$
\If{$D + \delta^2_n = D$ (up to double precision)}
\State \Return $\hat{\mathcal V}^* \gets \hat{\mathcal V}^n, \hat{\rho}^* \gets \hat{\rho}^n$
\Else
\State $D \gets D + \delta_n^2$
\EndIf
\EndFor
\State \Return $\hat{\mathcal V}^* \gets \hat{\mathcal V}^{n_{iters}}, \hat{\rho}^* \gets \hat{\rho}^{n_{iters}}$
\end{algorithmic}
\end{algorithm}
In addition to the spatial grid, define a sequence of time points $\{t_\theta \}_{\theta=1}^{N_T}$, taken such that $t_{\theta+1} - t_{\theta} = \tau$.
We time-discretize the semi-discrete systems we obtained using the stability preserving second-order Runge-Kutta method \cite{shu1988efficient}. The fixed temporal spacing was taken to be:
$$
\tau << .01(h s_{max} + k (\alpha s_{max}^2 + \max \{|u_{min},u_{max} \} + w_{max}))
$$
so as to satisfy the CFL condition robustly on both the HJB and FK equations, albeit at the cost of some efficiency.  
Similarly to \cite{chevalier2015micro}, we define:
\begin{equation}
\begin{split}
&\delta_n^2:=  ... \\ 
&\sum_{\theta = 1}^{N_T} \sum_{i,j=1}^{N_x,N_x}||(\hat{\mathcal V}, \hat \rho)_{i,j}^n(t_\theta) - (\hat{\mathcal V}, \hat \rho)_{i,j}^{n-1}(t_\theta)||^2_2 hk\tau,
\end{split}
\end{equation}
where $(\hat{\mathcal V}, \hat \rho)^n$ is the solution of the discretized PDE system after the $n-$th backward-forward iteration described in Algorithm 1., which is subject to the following convergence criterion:
\begin{theorem}\cite{chevalier2015micro}
If
$$
\sum_{n=1}^\infty \delta_n^2 < \infty,
$$
\end{theorem}
then there is a pair $(\hat{\mathcal V}, \hat \rho)^*$ s.t. 
$
(\hat{\mathcal V}, \hat \rho)^{(n)} \rightarrow (\hat{\mathcal V}, \hat \rho)^*.
$
\begin{proof}
The proof follows using the triangle inequality and invoking the completeness of finite dimensional real vector spaces \cite{chevalier2015micro}.
\end{proof}

\begin{figure}
\label{fig:bulk_vel}
\centering
\includegraphics[width=.5\textwidth, trim={0 10 0 20}]{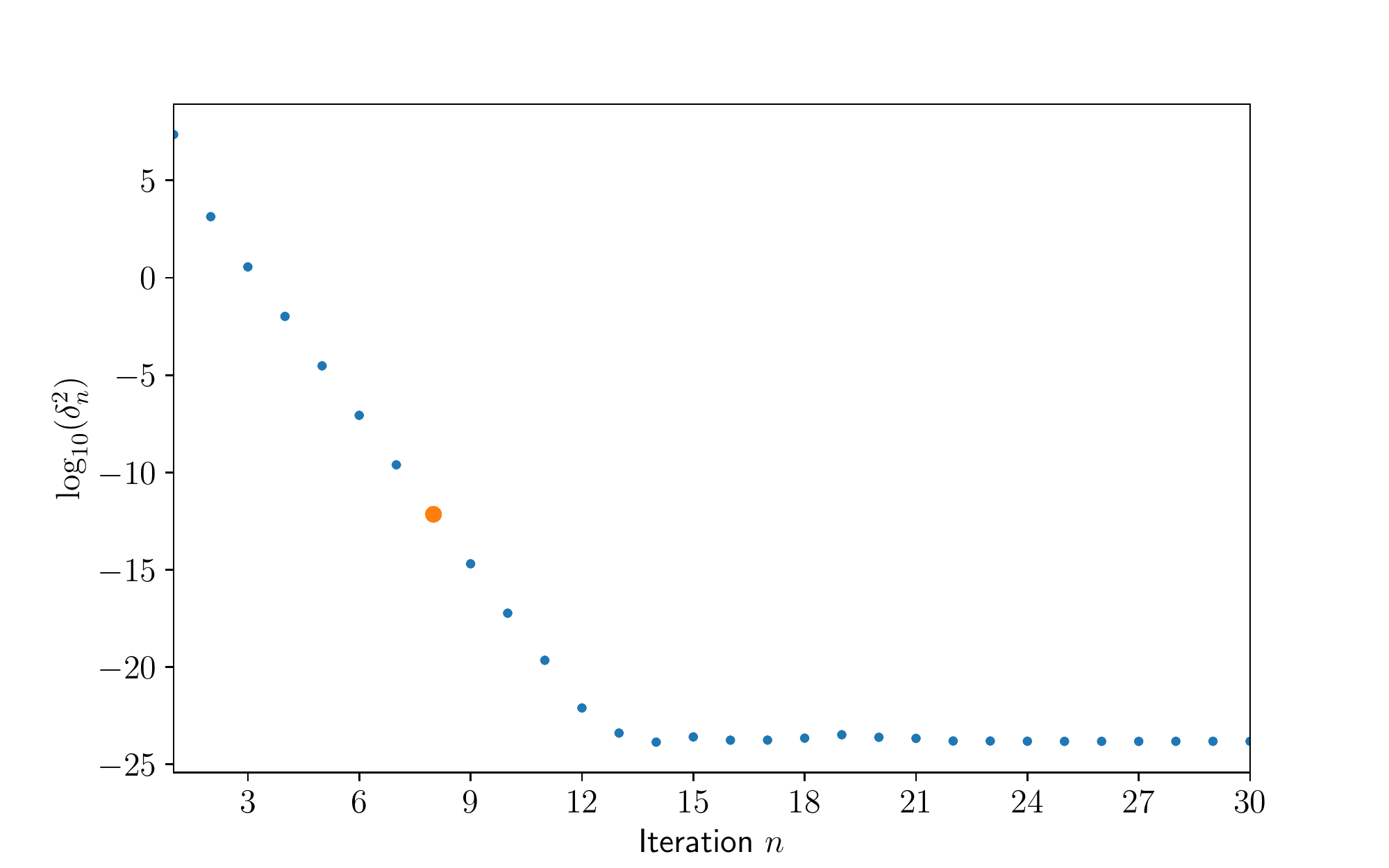}
\caption{The values of $\log_{10}(\delta_n^2)$ for the given numerical example. The orange point denotes the iteration at which the summed error ceases to change, up to double precision. For reference, we continue the iteration up to thirty backward-forward iterations.}
\end{figure}
\subsection{Model Parameters}
$\alpha, u_{max}, u_{min}$ are based on parameters for high-performance electric vehicles available in the US. $w_{max}, \epsilon$ were taken numerically to be twenty percent and half a percent, respectively, of $|u_{min}| > u_{max}$. $\alpha:= \frac{1}{2}\rho_{air}Ac_dm_{veh}^{-1} = \frac{1}{2}(1.2 \text{kg }\text{m}^{-3})(2.16 \text{m}^2)(.3)(1800 \text{ kg})^{-1} = 2.16 \cdot 10^{-4} \text{m}^{-1}$. $\epsilon = .05 \text{m}^2 \text{s}^{-1}$, $u_{max} = 8 \text{ m } \text{s}^{-2}$, $u_{min} = -10 \text{ m } \text{s}^{-2}$, $w_{max} = 2 \text{ m } \text{s}^{-2}$, $s_{max} = 30 \text{ m } \text{s}^{-1}$. $\gamma=.25$ and is unitless. We took $\beta = 4 \text{ s}^{3} \text{ m}^{-1}$, and $T= 30 \text{ s}$.

The congestion kernel $\phi$ is taken to be:
$$
\phi(\xi,\sigma):= \frac{1}{100}e^{\cos{((\xi-\sigma)/100)}} \text{ m } \text{s}^{-3},
$$
and the intial (unitless) position-velocity probability density:
$$
\rho_0(\xi,\upsilon):= \frac{1}{C}e^{\cos((\xi-100\pi)/100)}e^{-\frac{1}{2}(\upsilon - 20)^2}
$$
where $C$ is a normalizing constant. The initial density can be seen to be a product of a von Mises distribution and a variety of truncated normal distribution. As noted previously, $\mathcal V(T,\cdot) = 0$. For the numerical example we solve, the convergence of the backward-forward iterative method is displayed in Fig. 1.
\subsection{Simulation Parameters}
We take $N_x = 100$ for a $100 \text{x} 100$ grid staggered as described in $\upsilon$ and unstaggered in $\xi$. The time-spacing was taken to be $\tau = .001 \text{ s}$, the spatial increment was $h:= 200\pi/N_x \text{ m}$, and the velocity increment was $k := s_{max}/N_x \text{ m } \text{s}^{-1}$. The simulation was performed using a parallelized Fortran program called from Python.

\section{Results}
\begin{figure}
\centering
\includegraphics[width=.5\textwidth, trim={0 0 0 20}]{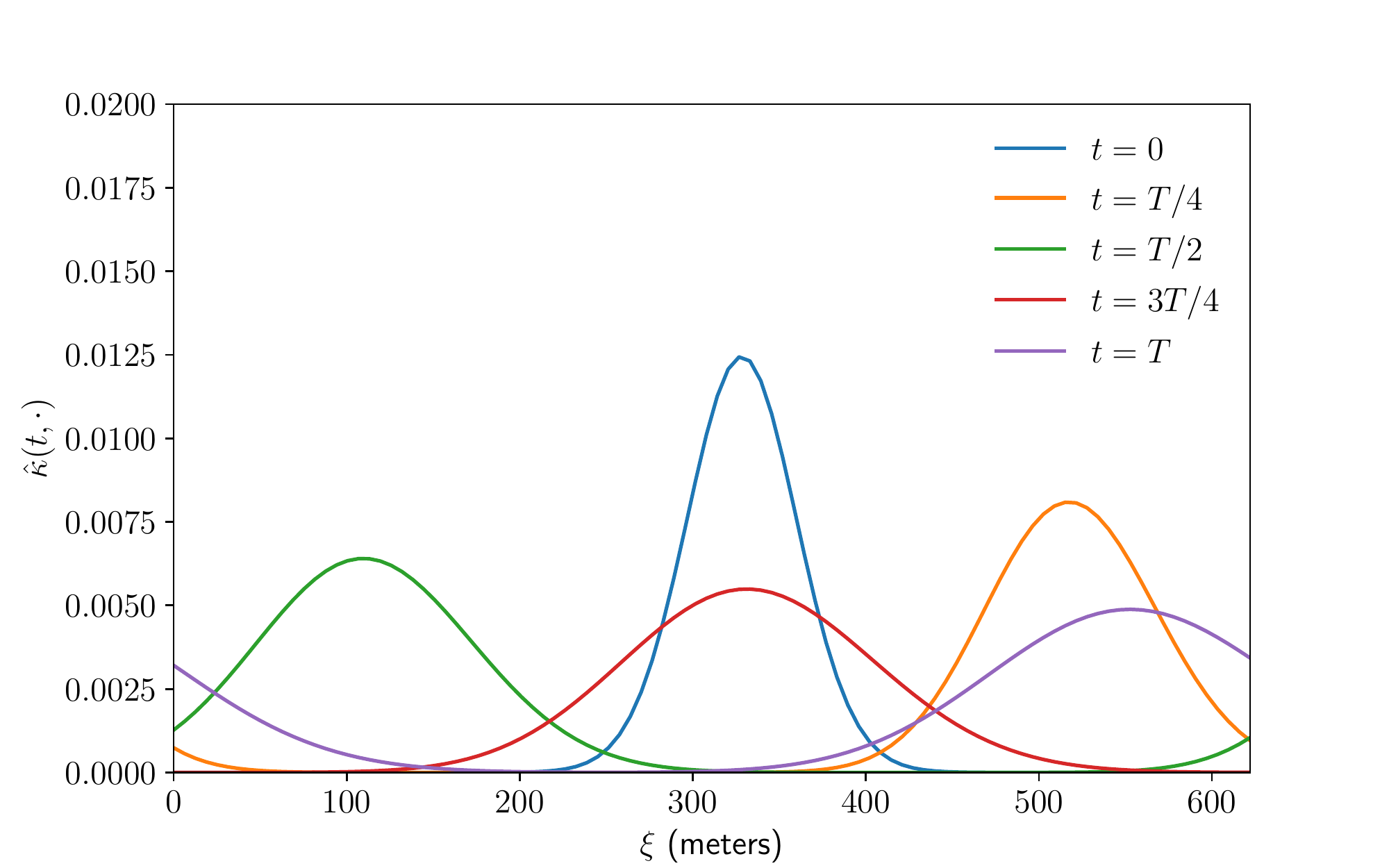}
\centering
\includegraphics[width=.5\textwidth, trim={0 10 0 34}]{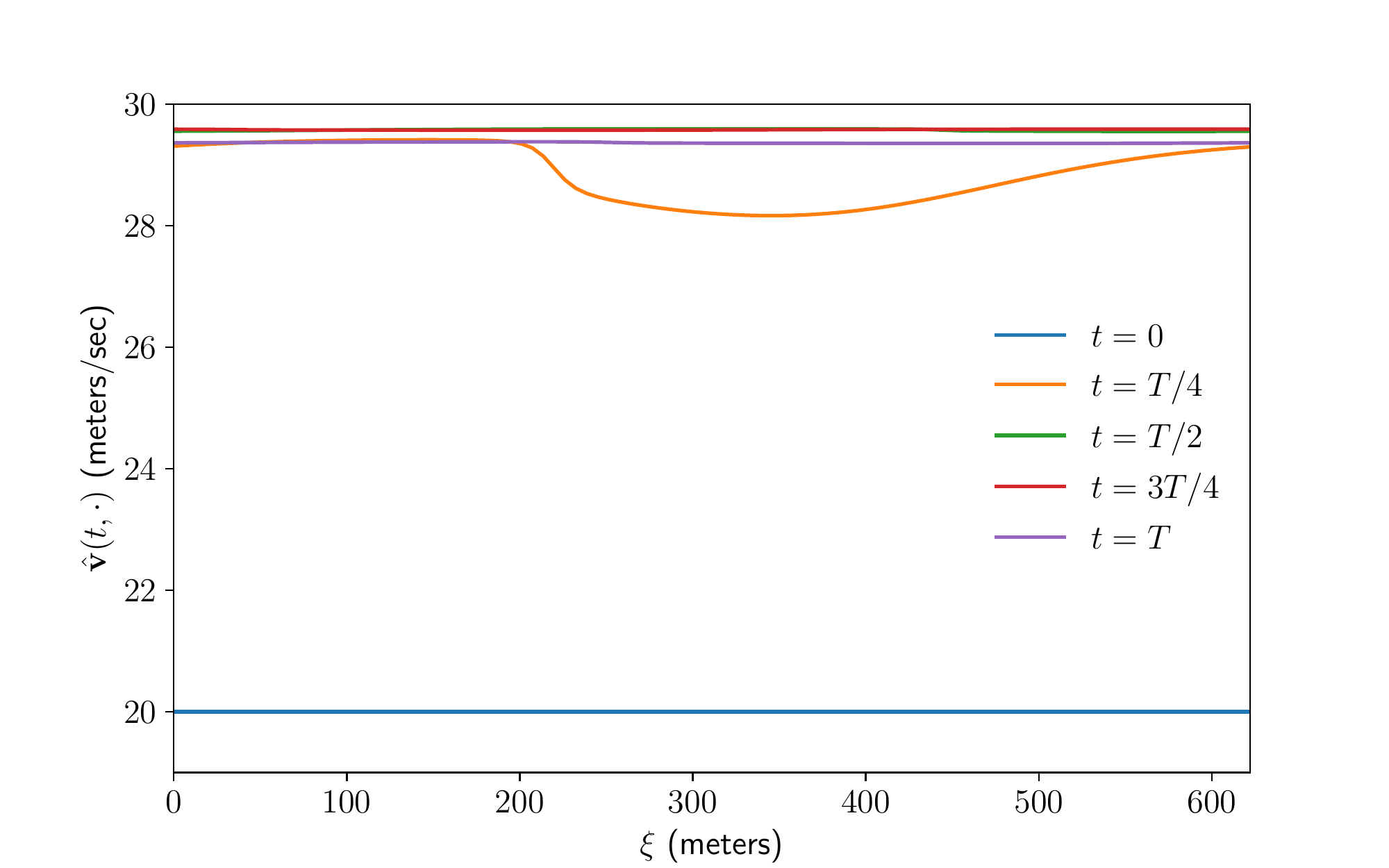}
\caption{(top) The spatial marginal probability distribution of the traffic. (bottom) The bulk velocity of the traffic flow. As the optimal control is applied, the congestion is dissipated. A slowdown region is created as the traffic approaches the congested region to rapidly dissipate congestion. Then, the slowdown region is itself dissipated. }
\end{figure}
We have approximated the density and associated value function of the agents as they follow the optimal control and worst-case disturbance. To interpret the results, we will motivate them by partially passing to an hydrodynamic limit of the kinetic equation.
Define the spatial marginal probability density $\kappa:[0,T] \times \mathbb T \rightarrow \mathbb R_0^+$, and the momentum density $\mathbf j:[0,T] \times \mathbb T \rightarrow \mathbb R$:
$$
\kappa (t,\cdot):= \int_\mathbb R \rho(t,\cdot) d\upsilon,
$$
$$
\mathbf j(t,\cdot):= \int_{\mathbb R} \upsilon \rho(t,\cdot) d\upsilon  =: \kappa (t,\cdot) \mathbf v(t,\cdot),
$$
assuming that $\rho(t,\cdot)$ is extended by $0$ for $\upsilon \notin [0,s_{max}]$.
Integrating the forward Kolmogorov equation over $\upsilon \in \mathbb R$, we obtain:
$$
\partial_t \kappa + \partial_\xi \mathbf (\kappa \mathbf v ) = 0 \text{ in } (0,T] \times \mathbb T
$$
subject to $\kappa(t,\xi) = \kappa(t,\xi + 200 \pi)$, and initial conditions obtained by integrating $\rho_0$. $\mathbf j$ appears as the first moment of the velocity from the forward Kolmogorov equation.  We compute approximate bulk velocity $\hat {\mathbf v}(t,\cdot)$ numerically (as a Riemann sum) from $\hat \rho(t,\cdot)$, and plot it in Fig. 2 (bottom). We also plot the spatial marginal probability density $\hat \kappa$ in Fig. 2 (top) also computed as a Riemann sum. $\hat{\mathbf v}$ is evolved in time so that the vehicles slow down before the congested region. This allows the vehicles sufficiently ahead of the congested region to speed up and spread into the rarefied regions of the road. As this process continues, the bulk velocity becomes more regular. Although the vehicles display a strong tendency to speed up, the drag and disturbances they are subject to prevents the bulk velocity from exactly matching the speed limit.

The smoothing of the bulk velocity is desireable as it indicates that there is not excessive acceleration or braking. This translates to fuel efficient and comfortable travel. 
\newline
\section{Conclusion}

We have developed and solved numerically a robust mean-field game on traffic flow. There are a number of avenues oringinating from this idea to explore. First, we intend to explore the $N$-player game and show  that the optimal control obtained herein is $\epsilon(N)$-optimal for the finite size game, with $\epsilon(N)$ monotonically decreasing with the number of players.

There is also the ergodic problem for this mean-field game. Apart from the mathematical aspects, the ergodic problem lends itself well to reinforcement learning. Then, one could obtain a congestion-dissipating adaptive control for the forward Kolmogorov equation online.

\bibliographystyle{IEEEtran}
\bibliography{biblio}

\end{document}